\documentclass{amsart}

\usepackage{graphicx} 
\usepackage{tikz}

\usepackage{booktabs} 
\usepackage{array} 
\usepackage{paralist} 
\usepackage{verbatim} 
\usepackage{subfig} 

\usepackage{amsmath}
\usepackage{amssymb}
\usepackage{color}
\usepackage{mathdots}
\usepackage{amsfonts}
\usepackage{mathrsfs}
\usepackage{mathtools}

\usepackage{amsthm}
\usepackage[section]{placeins}
\usepackage[colorlinks=true,linkcolor=black,citecolor=black]{hyperref}
\usepackage{enumitem}
\usepackage{graphicx}
\usepackage{pifont}
\usepackage{longtable}

\newcommand{\cmark}{\text{\ding{51}}}
\newcommand{\xmark}{~}

\usepackage{fancyhdr} 
\newtheorem{thm}{Theorem}[section]
\newtheorem{cor}[thm]{Corollary}

\newtheorem{lem}[thm]{Lemma}
\newtheorem{prop}[thm]{Proposition}
\newtheorem{rem}[thm]{Remark}

\title{Self-dual, self-Petrie-dual and M\"{o}bius regular maps on linear fractional groups.}

\begin{document}

\author{}
\date{}
\maketitle

\begin{center}

{\large Grahame Erskine} \\
\vspace{1.5mm} {\small Open University, Milton Keynes, U.K.}\\

grahame.erskine@open.ac.uk

\vspace{5mm}

{\large Katar\'{i}na Hri\v{n}\'{a}kov\'{a}} \\
\vspace{1.5mm} {\small
Slovak University of Technology, Bratislava, Slovakia}

hrinakova@math.sk

\vspace{3mm}

{and}

\vspace{3mm}

{\large Olivia Jeans} \\
\vspace{1.5mm} {\small Open University, Milton Keynes, U.K.}\\

olivia.jeans@open.ac.uk

\vspace{4mm}

\end{center}

\begin{abstract}
Regular maps on linear fractional groups $PSL(2,q)$ and $PGL(2,q$) have been studied for many years and the theory is well-developed, including generating sets for the asscoiated groups. This paper studies the properties of self-duality, self-Petrie-duality and M\"{o}bius regularity in this context, providing necessary and sufficient conditions for each case. We also address the special case for regular maps of type (5,5). The final section includes an enumeration of the $PSL(2,q)$ maps for $q\le81$ and a list of all the $PSL(2,q)$ maps which have any of these special properties for $q\le49$.

\end{abstract}

\section{Introduction}

Regular maps always display inherent symmetry by virtue of their definition. A regular map can have further symmetry properties which are called {\em external symmetries}. These occur when a map is isomorphic to its image under a particular operation. The best known example of this is the tetrahedron, a Platonic solid which is self-dual.

A {\em map} is a cellular embedding of a graph on a surface and is made up of vertices, edges and faces. We define a {\em flag} of the map to be a triple incidence of a vertex, an edge and a face. Informally we can visualise each flag as a triangle with its corners at the vertex, the centre of the face and the midpoint of the edge, and so the whole surface is covered by flags. We consider the symmetries of a map by reference to its flags. An {\em automorphism} of a map is an arbitrary permutation of its flags such that all adjacency relationships of the flags are preserved. The map is {\em regular} if the group of automorphisms acts regularly on the flags, that is the group is fixed-point-free and transitive. An implication of this is that each vertex of a regular map has a given valency, say $k$, and the face lengths are all equal, say to $l$. Henceforth we will refer to maps of type $(k,l)$ where $k$ is the vertex degree and $l$ is the face length of the regular map.

For further details about the theory of regular maps see \cite{{Brah}, {JS}, {Mac}, {Sah}, {Si-surv}}.

Every regular map has an associated {\em dual map} which is also a regular map. Informally, the dual map is created by forming a vertex at the centre of each original face and considering each of the original vertices as the centre of a face. Each edge of the dual map is thereby formed by linking a pair of neighbouring vertices across one of the original edges.

A different type of dual, the {\em Petrie dual} of a map has the same edges and vertices as the original map but the faces are different. That is, the underlying graph is the same, but the embedding is different. The boundary walk of a face of the Petrie dual map can be described informally as follows:
\begin{enumerate}
\item Starting from a vertex on the original map, trace along one side of an incident edge until you get to the midpoint of that edge;
\item Cross over to the other side of the edge and continue tracing along the edge in the same direction as before. When you approach a vertex, sweep the corner and continue along the next edge until you reach its midpoint;
\item Repeat step 2 until you rejoin the face boundary walk where you started.
\end{enumerate}

When the associated dual map is isomorphic to the original map, we call the map {\em self-dual} or {\em self-Petrie-dual} respectively. This paper explores necesary and sufficient conditions for a regular map with automorphism group $PSL(2,q)$ or $PGL(2,q)$, where $q$ is odd, to have each of these external symmetries.

Another property of interest in the theory of regular maps is {\em M\"{o}bius regularity}. This concept was introduced by S. Wilson in \cite{Wilson} who originally named them {\em cantankerous}. A regular map is M\"{o}bius regular if any two distinct adjacent vertices are joined by exactly two edges and any open set supporting these edges contains a M\"{o}bius strip. Clearly such a map must have even vertex degree $k$ and we will establish the further conditions under which a regular map on $PSL(2,q)$ or $PGL(2,q)$ is M\"{o}bius regular.

In Section 2 we state some of the background material and results which we will need.
Section 3 investigates regular maps of type $(p,p)$, $(k,p)$ and $(p,l)$  when $k$ and $l$ are coprime to $p$. Type $(p,p)$ is self-dual but not self-Petrie-dual nor M\"{o}bius regular, and type $(p,l)$ can be self-Petrie-dual but not M\"{o}bius regular. Type $(k,p)$ can be self-Petrie-dual or M\"{o}bius regular.
In Section 4 we address the necessary and sufficient conditions for a map of type $(k,l)$ to be self-dual, self-Petrie-dual and M\"{o}bius regular respectively.
Section 5 highlights a special case, namely maps of type $(5,5)$ whose orientation-preserving automorphism groups turn out to be isomorphic to $A_5$ and Section 6 comments on and lists examples of maps with some or all of these properties.

\bigskip

\section{Background information, notes and notation}

This paper is founded on work done by M. Conder, P. Poto\v{c}nik and J. \v{S}ir\'{a}\v{n} in \cite{CPS2} which provides a detailed analysis of reflexible regular hypermaps for triples $(k,l,m)$ on projective two-dimensional linear groups including explicit generating sets for the associated groups. In particular this paper is concerned only with maps, not hypermaps, and so, without loss of generality, we let $m=2$.

The group of automorphisms of a regular map is generated by three involutions two of which commute, where the three involutions can be thought of as local reflections in the boundary lines of a given flag which preserves all the  adjacency relationships between flags. As shown in Figure \ref{fig1} the involutions act locally on the given flag as follows: $X$ as a reflection in the edge bisector; $Y$ as a reflection across the edge; $Z$ as a reflection in the angle bisector at the vertex. The dots on the diagram indicate where there may be further vertices, edges and faces while the dashed lines outline each of the flags of this part of the map.

\begin{figure}
\begin{tikzpicture}
  \draw [inner sep=3pt, fill=gray, dotted] (0,0) -- (-1.5,0) -- (0,2.5) -- (0,0);
  \draw [line width=2pt] (1.5,0) -- (2.5,1.5);
  \draw [line width=2pt] (1.5,0) -- (2.5,-1.5);
  \draw [line width=2pt] (-1.5,0) -- (-2.5,1.5);
  \draw [line width=2pt] (-1.5,0) -- (-2.5,-1.5);
  \draw [line width=2pt] (-1.5,0) -- (1.5,0);
  \draw [dotted] (0,2.5) -- (2.5,1);
  \draw [dotted] (0,2.5) -- (-2.5,1);
  \draw [dotted] (0,-2.5) -- (2.5,-1);
  \draw [dotted] (0,-2.5) -- (-2.5,-1);
  \draw[dotted] (0,-2.5) -- (0,-2.5);
  \draw[dotted] (-1.5,0) -- (0,-2.5);
  \draw[dotted] (1.5,0) -- (0,2.5);
  \draw[dotted] (1.5,0) -- (0,-2.5);
  \draw[dotted] (0,-2.5) -- (0,0);
  \draw (-3,-0.3) node[circle,fill,inner sep=1pt]{};
  \draw (-3.1,0) node[circle,fill,inner sep=1pt]{};
  \draw (-3,0.3) node[circle,fill,inner sep=1pt]{};
  \draw (3,-0.3) node[circle,fill,inner sep=1pt]{};
  \draw (3.1,0) node[circle,fill,inner sep=1pt]{};
  \draw (3,0.3) node[circle,fill,inner sep=1pt]{};
  \draw [dotted] (0,0) -- (-1.5,0) -- (0,2.5) -- (0,0);
  \draw [<->] (-0.3, 1) -- (0.3, 1) node[above] {X};
  \draw [<->] (-0.75, -0.3) node[right] {Y} -- (-0.75, 0.3);
  \draw [<->] (-0.7, 0.8) -- (-1.2, 1.1) node[above] {Z};
 \end{tikzpicture}
\caption{The action of automorphisms X, Y and Z on the shaded flag}
\label{fig1}
\end{figure}
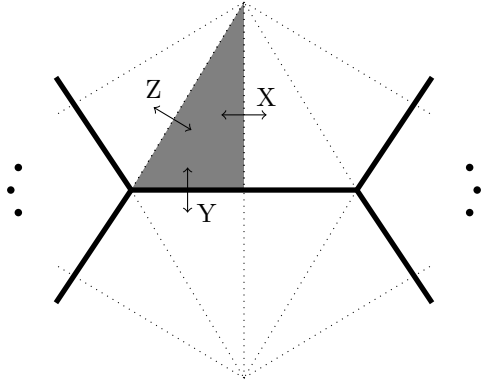

The study of regular maps is equivalent to the study of group presentations of the form $G\cong \langle X,Y,Z | X^2, Y^2, Z^2, (YZ)^k, (ZX)^l, (XY)^2, ... \rangle$, see \cite{Si-surv}. The dots indicate the potential for further relations not listed, and we assume the orders shown are indeed the true orders of those elements in the group.

The surface on which a regular map is embedded could be orientable or non-orientable. If the regular map is on an orientable surface then $G$ has a subgroup of index two which corresponds to the orientation preserving automorphisms. Instead of the group generated by these three involutions $X, Y$ and $Z$, we can consider the group of orientation-preserving automorphisms which is generated by the two rotations $R=YZ$ and $S=ZX$. On a non-orientable surface these two elements will still generate the full automorphism group and we can say that studying these maps is equivalent to studying groups which have presentaions of the form $\langle R, S| R^k, S^l, (RS)^2,...\rangle$.

We focus on regular maps of type $(k,l)$ where the associated group $G \cong \langle X,Y,Z \rangle$ is isomorphic to $PSL(2,q)$ or $PGL(2,q)$ where $q$ is a power of a given odd prime $p$.
Following the convention and notation established in \cite{CPS2}, and applying it to regular {\em maps}, we let $\xi_\kappa$ and $\xi_\lambda$ be respectively primitive $2k$th and $2l$th roots of unity in the finite field $GF(q)$ or $GF(q^2)$ and we define $\omega_i = \xi_i + \xi_i^{-1}$ for $i \in \{\kappa,\lambda\}$. We too assume that $(k,l)$ is a {\em hyperbolic} pair, that is $1/k+1/l < 1/2$. This implies that $k \ge3$ and $l \ge3$. The conditions in this paragraph are what we refer to as {\em the usual setup}.

We can consider duality and Petrie duality as operators on a map. Since the dual of a map is obtained by swapping the vertices for faces and vice versa, in terms of the involutions $X, Y, Z$ the dual operator would fix $Z$ and interchange $X$ and $Y$. The Petrie dual operator would replace $X$ with $XY$ and fix $Y$ and $Z$.  The automorphism associated with any type of duality is an involution. This is because it acts on our map to produce the dual map, and when this automorphism is repeated we get back to the original map. Self-duality and self-Petrie-duality are therefore equivalent to the existence of precisely such involutory automorphisms of $G$, the group associated with the regular map.

Our paper is devoted in large part to finding conditions for the existence of involutory automorphisms which imply self-duality and/or self-Petrie-duality.
The automorphism group for $G$ is $P\Gamma L(2,q)$, the semidirect product $PGL(2,q) \rtimes C_e$ where $q=p^e$. Elements $(A,j) \in P\Gamma L(2,q)$ act as follows: $(A,j)(T) = A \phi_j (T) A^{-1}$ where $\phi_j$ is the repeated Frobenius field automorphism of the finite field, $\phi_j : x \to x^r$ with $r=p^j$. The function $\phi_j$ acts element-wise on a matrix and we use the general rule for compostion in $P\Gamma L(2,q)$ which is $(B,j)(A,i) = (B \phi_j (A), i+j)$. 

When $(A,j) \in P\Gamma L (2, p^e)$ is an involution, it must be such that $(A,j)(A,j) = (A \phi_j (A), 2j ) $ is the identity, so $2j \equiv 0$ (mod e). One case is when there is no field automorphism involved, that is $j=0$ and $A^2 = I$. Alternatively $e = 2j$ is even, and then we need $\phi_j (A) = A^{-1}$. This is summarised in Lemma \ref{invo}.

\begin{lem}\label{invo}
$(A,j) \in P\Gamma L (2, p^e)$ is an involution if and only if one of the following conditions holds:
\begin{enumerate}
\item $j=0$ and $A^2=I$
\item $2j=e$ and $\phi_j (A) = A^{-1}$.
\end{enumerate}
\end{lem}

Explicit generating sets are known for regular maps with automorphism group $G$ isomorphic to $PSL(2,q)$ or $PGL(2,q)$, and for details we refer the interested reader to \cite{CPS2}. We present the results for maps of each type as required.

We will need to consider performing operations on the elements $X$, $Y$ and $Z$ of $G$. As such we denote elements of the group $PSL(2,q)$ or $PGL(2,q)$, by a representative matrix with square brackets. This allows us to perform the necessary calculations. We can then determine whether or not two resulting matrices are equivalent within $G$, that is whether or not they correspond to the same element of the group $G$. A pair of matrices are in the same equivalence class, that is they represent the same {\em element} of $G$, if one is a scalar multiple of the other. We use curved brackets for matrix representatives for $X$, $Y$ and $Z$.

Lemma \ref{invo} can then be used to find conditions for the elements of the matrix part $A$ of an involutory automorphism $(A,j)$ as follows.

\begin{lem}\label{invoabcd}
Let $A = \begin{bmatrix}a & b \\ c & d\end{bmatrix}$. The automorphism denoted $(A,j) \in P\Gamma L(2,p^e)$ is an involution, if and only if $a, b, c, d$ satisfy the following equations, with $r=p^j$ for $j=0$ or $2j=e$.
\begin{enumerate}
\item $a^{r+1} = d^{r+1}$, 
\item $bc^r = c b^r$, 
\item $ab^r + bd^r = 0$, 
\item $ca^r + dc^r = 0$. 
\end{enumerate}
\end{lem}

\begin{proof}
By Lemma \ref{invo}, and letting $r=p^j$ we have $A\phi_j(A) = \begin{bmatrix}a & b \\ c & d\end{bmatrix} \begin{bmatrix}a^r & b^r \\ c^r & d^r\end{bmatrix} = \begin{bmatrix}a^{r+1} +b c^r & ab^r +b d^r \\ ca^r +dc^r & cb^r+d^{r+1}\end{bmatrix} = I$. By comparing the leading diagonal entries we see that $a^{r+1} +b c^r = cb^r+d^{r+1}$. Applying the field automorphism $\phi_j$ 
yields $a^{1+r} +b^r c= c^r b+d^{1+r}$. Subtracting these two equations, and remembering that $q$ is odd, we get the first two equations, while looking at the off-diagonal immediately gives rise to the final two equations.
\end{proof}

When we are establishing the conditions under which a regular map is M\"{o}bius regular we will rely on the following group-theoretic result, proved in \cite{LiS} by Li and \v{S}ir\'{a}\v{n}. Note that implicit in this necessary and sufficient condition is that for a map of type $(k,l)$ to be M\"{o}bius regular $k$ must be even.

\begin{lem}\label{MR}
A regular map is M\"{o}bius regular if and only if $XR^{\frac{k}{2}}X=R^{\frac{k}{2}}Y$.
\end{lem}

\bigskip 

\section{Regular maps on linear fractional groups of type $(p,p)$, $(k,p)$ and $(p,l)$ where $p$ is an odd prime}

For odd prime $p$, by Proposition 3.1 in \cite{CPS2}, maps of the type $(p,p)$ have the following representatives for $X$, $Y$ and $Z$, where $\alpha^2=-1$:
$$
X_1=
- \alpha \left( \begin{array}{cc}
1 & 0\\
2 & {-1}
\end{array} \right)
\text{ , }
Y_1=
- \alpha
\left( \begin{array}{cc}
1 & -1\\
0 & -1
\end{array} \right)
\text{ and }
Z_1=
\alpha
\left( \begin{array}{cc}
1 & 0\\
0 & -1
\end{array} \right)
$$

\begin{prop}
With the usual setup, a map of type $(p,p)$ is self-dual.
\end{prop}

\begin{proof}
For self duality we need $G\cong \langle X,Y,Z \rangle$ to admit an automorphism such that $X$ and $Y$ are interchanged, and $Z$ is fixed. So the question is: can we find an automorphism $(A,j) \in P\Gamma L(2,q)$ such that $A\phi_j(X)A^{-1} = Y$, $A\phi_j(Y)A^{-1} = X$, and  $A\phi_j(Z)A^{-1} = Z$. It is easy to verify that $(A, 0)$, where $A$ has the form $A= \begin{bmatrix} 0&1 \\ 2&0 \end{bmatrix}$ satisfies these conditions, so this type of map is self-dual.
\end{proof}

\begin{prop}
With the usual setup, a map of type $(p,p)$ is not self-Petrie-dual.
\end{prop}

\begin{proof}
In order to be self-Petrie-dual, the group $G$ needs to admit an involutory automorphism $(B,j)$ which fixes $Z$ and $Y$, and exchanges $X$ with $XY$.

First notice that $\phi_j (Z) = Z$ and $\phi_j (Y) = Y$ so if $B$ exists, it must be of a form which commutes with both $Z$ and $Y$. To commute with $Z$, the necessarily non-identity element $B$ must be either $B_1= \begin{bmatrix} 0&b \\ c&0 \end{bmatrix}$ or $B_2= \begin{bmatrix} a&0 \\ 0&d \end{bmatrix}$. Note that $0 \notin \{a,b,c,d\}$ and $a\neq d$. As shown below, neither of these commute with $Y$.
$$B_1 Y =
 - \alpha \begin{bmatrix} 0 & -b \\ c & -c \end{bmatrix} \neq
Y B_1 =
-\alpha \begin{bmatrix} -c & b \\ -c & 0 \end{bmatrix}
$$

$$B_2 Y =
 - \alpha \begin{bmatrix} a & -a \\ 0 & -d \end{bmatrix} \neq
Y B_2 = 
-\alpha \begin{bmatrix} a & -d \\ 0 & -d \end{bmatrix}
$$

 Hence this type of map is not self-Petrie-dual.
\end{proof}

\begin{rem}
Maps of type $(k,p)$ and $(p,l)$ where $k$ and $l$ are coprime to $p$ clearly cannot be self-dual since the vertex degree and face lengths differ.
\end{rem}

\begin{prop}
With the usual setup, and for $k$ coprime to $p$, a map of type $(k,p)$ is self-Petrie-dual if and only if $k|2(r \pm 1)$ and $\pm \omega_\kappa^{(r + 1)} = 4 \xi_\kappa^{(r\pm 1)}$ when the corresponding signs in each $(r\pm 1)$ are read simultaneously, and where $r=p^j$ and $j=0$ or $2j=e$.
\end{prop}

\begin{proof}
When $k$ is coprime to $p$, \cite{KH} tells us that a map of type $(k,p)$ has the following triple of generating matrices corresponding to $X$, $Y$, and $Z$:
$$
X_2= \eta\alpha\left(
\begin{array}{cc} -\omega_{\kappa} &  -2\xi_{\kappa}\\ 2\xi_{\kappa}^{-1} &
\omega_{\kappa}\end{array}\right),
Y_2 =-\alpha\left(\begin{array}{cc} 0 &
\xi_{\kappa}
\\ \xi_{\kappa}^{-1} & 0\end{array}\right),
Z_2 = \alpha\left(\begin{array}{cc} 0 & 1 \\ 1 & 0
\end{array}\right),
$$

\noindent where $\alpha^2=-1$ and $\eta=(\xi_{\kappa}-\xi_{\kappa}^{-1})^{-1}$.

Suppose the map is self-Petrie-dual and $(B,j) = (\begin{bmatrix} a & b\\ c&d\end{bmatrix},j)$ is the associated involutory automorphism. In order to fix $Z$ we must have $a=d$ and $b=c$ or $a=-d$ and $b=-c$.
In order to fix $Y$ we find that either $a=0$ or $c=0$ in which case we need $k| 2(p^j + 1)$ or $k| 2(p^j -1)$ respectively. When a=0, the involution then interchanges $X$ with $XY$ if and only if $\pm \omega_\kappa^{(r+1)} = 4 \xi_\kappa^{(r+1)}$. When c=0, the involution then interchanges $X$ with $XY$ if and only if $\pm \omega_\kappa^{(r+1)} = 4 \xi_\kappa^{(r-1)}$. 
\end{proof}

\begin{prop}
Under the usual setup, and with $l$ coprime to $p$, a regular map of type $(p,l)$ is self-Petrie-dual if and only if $\omega_\lambda^2 = -\omega_\lambda^{2r}$ where $r=p^j$ and $2j= e$.
\end{prop}

\begin{proof}
Using a similar argument to the above applied to the appropriate matrix triple from \cite{KH}, namely
$$
X_3 = \alpha\left(
\begin{array}{cc} 0& \omega_{\lambda}^{-1} \\
\omega_\lambda & 0\end{array}\right), 
Y_3 =-\alpha\left(\begin{array}{cc} 1 & 0 \\
0 & -1 \end{array}\right), 
Z_3 = \alpha\left(\begin{array}{cc} 1 & 1 \\ 0 & -1
\end{array}\right),
$$ we find that the only allowable form for $B$ is the identity. The necessary non-trivial field automorphism applied to the X and XY interchange then yields the stated condition.
\end{proof}

\begin{rem}
For odd $p$, a regular map of type $(p,p)$ or $(p,l)$ is not M\"{o}bius regular.
This is immediate from the fact that each pair of adjacent vertices in a M\"{o}bius regular map is joined by exactly two edges, hence the vertex degree must be even.
\end{rem}

\begin{prop}
Under the usual setup for even $k$, a map of type $(k,p)$ is M\"{o}bius regular if and only if $\omega_\kappa^2 + 4 =0$
\end{prop}

\begin{proof}

A regular map is M\"{o}bius regular if and only if the equation $XR^{\frac{k}{2}}X=R^{\frac{k}{2}}Y$ is satisfied. We assume $k$ is even, since if $k$ is odd then the map is certainly not M\"{o}bius regular. In this case $R= [Y_2 Z_2] =\begin{bmatrix} \xi_{\kappa}&0\\ 0&\xi_{\kappa}^{-1} \end{bmatrix}$ so we have $R^{\frac{k}{2}}=\begin{bmatrix} \alpha&0\\ 0&\alpha^{-1} \end{bmatrix}$, where $\alpha^2=-1$.
Hence the map is M\"{o}bius regular if and only if these matrices are equivalent:
\begin{eqnarray*}
XR^{\frac{k}{2}}X  =  -\eta^2\alpha\begin{bmatrix} \omega_{\kappa}^2+4 & 4\omega_\kappa \xi_\kappa \\ -4 \omega_\kappa \xi_\kappa^{-1} & -(\omega_{\kappa}^2+4) \end{bmatrix}  \text{  and  }   R^{\frac{k}{2}}Y  =  \begin{bmatrix} 0&\xi_{\kappa}\\ -\xi_{\kappa}^{-1}&0 \end{bmatrix}.\\
\end{eqnarray*}

These matrices are equivalent if and only if $\omega_\kappa^2 + 4 =0$.
\end{proof}

\bigskip

\section{Regular maps on linear fractional groups of type $(k,l)$ where both $k$ and $l$ are coprime to $p$}

In this case we have different generating triples for the group $G$. As per Proposition 3.2 in \cite{CPS2}, the triple $(X, Y, Z)$ has representatives as defined below where $D=\omega_\kappa^2+\omega_\lambda^2-4$, $\beta= -1/\sqrt{-D}$ and $\eta= (\xi_\kappa - \xi_\kappa^{-1})^{-1}$.

$
X_4=
\eta\beta \left( \begin{array}{cc}
D & D\omega_\lambda \xi_\kappa\\
-\omega_\lambda \xi_\kappa^{-1} & -D
\end{array} \right)
$
,
$
Y_4=
\beta
\left( \begin{array}{cc}
0 & \xi_\kappa D\\
{\xi_\kappa}^{-1} & 0
\end{array} \right)
$
, and 
$Z_4=
\beta
\left( \begin{array}{cc}
0 & D\\
1 & 0
\end{array} \right)
$ 

\smallskip
We will also consider the pair of matrices which represent $R$ and $S$, the rotations around a vertex and a face respectively, which by Proposition 2.2 in \cite{CPS2} are: 

$R_4 = \left( \begin{array}{cc} \xi_\kappa & 0 \\ 0 & \xi_\kappa^{-1} \end{array} \right)$
 and 
$ S_4 = \eta \left( \begin{array}{cc} - \omega_\lambda \xi_\kappa^{-1} & -D \\ 1 & \omega_\lambda \xi_\kappa \end{array} \right)$.

At this point we note that there is an exception for maps of type $(5,5)$, which is addressed in Section \ref{A5}.

\smallskip
\begin{thm}\label{klSD}
Under the usual setup, a map of type $(k,k)$ is self-dual if and only if
$\omega_\lambda = \pm \omega_\kappa^r$ where $r=p^j$, and $j=0$ or $2j=e$.
\end{thm}

\begin{proof}
Suppose the map is self-dual.

There is an involutory automorphism of $G$ which fixes $Z$ and interchanges $X$ and $Y$. This is equivalent to interchanging the rotations $R^{-1}=(YZ)^{-1}=ZY$ and $S=ZX$ around a vertex and a face respectively. That is there is an automorphism $(A,j)$ which interchanges $\pm R^{-1}$ with $S$. Here the $\pm$ takes into account both representative elements for $R$. So $A(\pm \phi_j (R^{-1}))A^{-1} = S$. Remembering that conjugation preserves traces this implies $\pm \phi_j \mathrm{tr}(R^{-1}) = \mathrm{tr}(S)$ which immediately yields the condition  $\pm \omega_\kappa^{r} = \omega_\lambda$.

Conversely suppose $\pm \omega_\kappa^{r} = \omega_\lambda$.

We note that $\omega_\kappa^{2r} = \omega_\lambda^2 \iff \omega_\kappa^2 = \omega_\lambda^{2r}$
and so $D^r = (\omega_\kappa^2 + \omega_\lambda^2 -4 )^r =\omega_\kappa^{2r} + \omega_\lambda^{2r} -4 = \omega_\lambda^2 + \omega_\kappa^2 -4 = D$.

We aim to find an involutory automorphism $(A,j)$ which demonstrates this map is self-dual. Consider $A = \begin{bmatrix}a & D \\ -1 & -a \end{bmatrix}$ which, by Lemma \ref{invoabcd}, so long as $a^r=a$, satisfies all the equations necessary for the element $(A,j)$ to be involutory. Notice that $(A,j)$ also fixes Z. We also need $X$ and $Y$ to be interchanged by the automorphism in which case the following matrices are equivalent.

$
A \phi_j (X) = \eta^r\beta^r\begin{bmatrix} D(a - \omega_\lambda^r\xi_\kappa^{-r}) & D( a \omega_\lambda^r \xi_\kappa^r -D)  \\  \omega_\lambda^r \xi_\kappa^{-r}a - D & D(a- \omega_\lambda^r \xi_\kappa^r) \end{bmatrix}
\text{ and } YA = \beta \begin{bmatrix} -\xi_\kappa D & - a\xi_\kappa D \\ a \xi_\kappa ^{-1} & D \xi_\kappa ^{-1} \end{bmatrix}
$
\smallskip

Ratio of elements in the leading diagonal: $ -\xi_\kappa^2 = { (a - \omega_\lambda^r\xi_\kappa^{-r})}/{(a- \omega_\lambda^r \xi_\kappa^r)} $

Ratio of elements in the left column: $-D \xi_\kappa^2 /a = { D(a - \omega_\lambda^r\xi_\kappa^{-r})}/{(\omega_\lambda^r \xi_\kappa^{-r}a - D)} $

Ratio of elements in the top row: $1/a = { (a - \omega_\lambda^r\xi_\kappa^{-r})}/{ ( a \omega_\lambda^r \xi_\kappa^r -D)} $

\smallskip
The last ratio listed yields a quadratic in a:
$0 = a^2 - a\omega_\lambda^r (\xi_\kappa ^{-r} + \xi_\kappa^r) + D = a^2   - a \omega_\lambda^r \omega_\kappa^r + D$. This is consistent with all the necessary ratios. All that remains is to check that a value of $a$ satisfying this quadratic is invariant under the repeated Frobenius field automorphism. The discriminant $\Delta = \omega_\kappa^2\omega_\kappa^{2r} - 4(\omega_\kappa^2 +(\omega_\kappa^r)^2 -4)=((\omega_\kappa^r)^{2} -4)(\omega_\kappa^2 - 4)$. Furthermore the expression for $a=(\omega_\lambda^r \omega_\kappa^r \pm \sqrt\Delta)/2$ is invariant under the transformation $x \to x^r$ as required. Hence the map is self-dual.
\end{proof}

\begin{thm}\label{klSPD}
With the usual setup, where $k,l$ are coprime to $p$, a map of type $(k,l)$ is self-Petrie-dual if and only if one of the following conditions is fulfilled:
\begin{trivlist}
\item 1. $\omega_\lambda^{2}=-D$
\item 2. $q=r^2=p^{2j}$, $\omega_\lambda^{2r}=-D$ and $k|r\pm1$.
\end{trivlist}
\end{thm}

\begin{proof}
First suppose the map is self-Petrie-dual. So there exists $(B,j) \in P\Gamma L(2,q)$ such that $B \phi_j (X) B^{-1} = XY$, $B \phi_j (Y) B^{-1} = Y$, and $B \phi_j (Z) B^{-1} = Z$. By comparing the traces of $\phi_j(ZX)$ and $ZXY$ we get  the necessary conditon: $\omega_\lambda^{2r} = -D$.

For the rest of the proof we split the situation into two cases: the first when $j=0$ and we do not consider any field automorphism, and the second case where a field automorphism is included.

\smallskip

Case 1: $j=0$.

Suppose $\omega_\lambda^2 = -D$.
Notice that $B = \begin{bmatrix} 1&0 \\ 0 & -1 \end{bmatrix}$ fixes both Y and Z. 
The map is self-Petrie-dual if
$
BX 
= \eta \beta \begin{bmatrix} D &  D\omega_\lambda \xi_\kappa \\  \omega_\lambda \xi_\kappa^{-1} & D \end{bmatrix}
$ and 
$XYB =
\eta \beta^2\begin{bmatrix} D\omega_\lambda & -D^2\xi_\kappa  \\ -D\xi_\kappa^{-1} & D \omega_\lambda \end{bmatrix}
$ are also equivalent. Comparing these and applying our assumption that $\omega_\lambda ^2 = -D$ we conclude this map is self-Petrie-dual.

\smallskip

Case 2: $2j=e$. By Lemma \ref{invo} we include the repeated Frobenius automorphism.

\smallskip

Suppose $\omega_\lambda^{2r}=-D$ and $k|r\pm1$.

The map is self-Petrie-dual if there is an involutory automorphism which not only fixes $Z$ but also fixes $Y$ and interchanges XY with X. We hope to find $(B,j)= (\begin{bmatrix} a & b \\ c & d \end{bmatrix},j)$, the associated element of $P\Gamma L$. In addition to the conditions for $a,b,c,d$ established in Lemma \ref{invoabcd}, we require $B \phi_j(Z) B^{-1} = Z$ and 
 $ B \phi_j(Y) B^{-1} = Y$.  In order to fix $Z$ we must have $bd = acD^r$ and $D(d^2 - c^2D^r)=a^2D^r-b^2$. Fixing $Y$ yields two further equations: $bd=ac\xi_\kappa^{2r}D^r$ and $\xi_\kappa^2D(d^2\xi_\kappa^{-r} - c^2\xi_\kappa^rD^r) = a^2\xi_\kappa^{r}D^{r} - b^2\xi_\kappa^{-r}$. Since $\xi_\kappa^{2r} \neq 1$ and $D \neq 0$, notice that $bd=ac\xi_\kappa^{2r}D^r = acD^r$ tells us that either $a=0$ or $c=0$.

If $a=0$, we immediately see $d=0$ too and so we can assume $b=1$ without loss of generality.
The equations for $a,b,c,d$ tell us that to fix $Z$ we have $c^2 = \frac{1}{D^{r+1}}$ and to fix Y we have $c^2 \xi^{2r+2} = \frac{1}{D^{r+1}}$. So this automorphism exists only if $\xi_\kappa^{2r+2} = 1$. By definition $\xi_\kappa$ is a primitive $2k$th root of unity and $\xi_\kappa^{2r+2} = 1 \iff 2k | 2r+2 \iff k| r+1$, which is the case by our assumption.

$B \phi_j (XY) = \eta^r \beta^{2r} \begin{bmatrix} -D^r \xi^{-r} & \mp D^r \sqrt{-D} \\ \pm c D^r \sqrt{-D} & c D^{2r} \xi_\kappa^r  \end{bmatrix}$ and
$XB = \eta \beta \begin{bmatrix} cD \omega_\lambda \xi_\kappa & D \\ -cD & -\omega_\lambda \xi_\kappa^{-1}  \end{bmatrix}$ are also equivalent if the map is self-Petrie-dual so we compare the ratios of the elements in turn. This yields $\pm c = \frac{1}{\omega_\lambda \xi_\kappa^{r+1} \sqrt{-D}} = \frac{\omega_\lambda\sqrt{-D}}{D^{r+1} \xi_\kappa^{r+1}}$ which is true only if $D^r = -\omega_\lambda^2$, which is again the case by our assumption. These conditions are consistent with our other requirements for the value of $c$, (namely that $c^r=c$) so we have an automorphism demonstrating that this map is self-Petrie-dual.

If on the other hand $c=0$ then we have $b=0$ and we assume $a=1$ without loss of generality. Fixing $Z$ yields $d^2D=D^r$. Fixing $Y$ yields $d^2D=\xi_\kappa^{2r-2}D^{r}$. So the map is self-Petrie-dual only if $\xi_\kappa^{2r-2} = 1$, which is the case by our assumption.

Now $B \phi_j (XY) = \eta^r \beta^{2r}D^r \begin{bmatrix}  \omega_\lambda^r & D^r\xi_\kappa^r \\ -d\xi_\kappa^{-r} & -d\omega_\lambda^r \end{bmatrix}$ and
$XB = \eta \beta \begin{bmatrix} D  & dD\omega_\lambda\xi_\kappa  \\ -\omega_\lambda\xi_\kappa^{-1} & -dD \end{bmatrix}$.

Again, the map is self-Petrie-dual if these two elements are equivalent, that is if $D^r\xi_\kappa^{r-1}= d\omega_\lambda^{r+1}$ and $\omega_\lambda^{r+1}\xi_\kappa^{r-1} = dD$. Applying our assumption $\omega_\kappa^{2r} = -D$, we conclude the map is self-Petrie-dual.
\end{proof}

The preceeding two results, Theorem \ref{klSD} and Theorem \ref{klSPD}, indicate a {\em sufficient} condition for a regular map of type $(k,k)$ to be both self-dual and self-Petrie-dual, namely $\omega_\kappa^2 = \omega_\lambda^2 = -D$. Corollary \ref{3omsq4} shows this becomes a tractable sufficient condition for both self-duality and self-Petrie-duality.

\begin{cor}\label{3omsq4}
If $\omega = \omega_\kappa = \omega_\lambda$ and $3\omega^2 = 4$ then the associated map is both self-dual and self-Petrie-dual.
\end{cor}

We now turn our attention to the conditions for M\"{o}bius regularity.

\begin{prop}
With the usual setup, a regular map of type $(k,l)$ is M\"{o}bius regular if and only if $k$ is even and $\omega_\kappa^2+2\omega_\lambda^2=4$.
\end{prop}

\begin{proof}
By Lemma \ref{MR}, a regular map is M\"{o}bius regular if and only if the equation $XR^{\frac{k}{2}}X=R^{\frac{k}{2}}Y$ is satisfied. In this case $R= \begin{bmatrix} \xi_{\kappa}&0\\ 0&\xi_{\kappa}^{-1} \end{bmatrix} $. So $R^{\frac{k}{2}}=\begin{bmatrix} \alpha&0\\ 0&-\alpha \end{bmatrix}$, where $\alpha^2=-1$. Then $XR^{\frac{k}{2}}X=R^{\frac{k}{2}}Y$ is satisfied if and only if
\begin{eqnarray*}
\eta^2 \beta^2 \begin{bmatrix} \alpha D^2 + \alpha D \omega_\lambda^2 & 2 \alpha D^2 \omega_{\lambda}\xi_{\kappa}\\ -2 \alpha D \omega_{\lambda}\xi_{\kappa}^{-1}&-\alpha D \omega_\lambda^2 - \alpha D^2 \end{bmatrix} &=&\beta\begin{bmatrix} 0&\alpha\xi_{\kappa}D\\ -\alpha \xi_{\kappa}^{-1}&0 \end{bmatrix}\ .\\
\end{eqnarray*}

The elements on the leading diagonal must be zero, which yields just one equation:
$\eta^2\beta \alpha D( D+\omega_{\lambda}^2)=0$. The ratio between the non-zero entries is the same for both matrices and so no further conditions arise.

We conclude that for even $k$, the map is M\"{o}bius regular if and only if $D=-\omega_{\lambda}^2$, which is equivalent to
$\omega_{\kappa}^2+2\omega_{\lambda}^2=4$.
\end{proof}

It is not surprising to see some similarity between conditions for self-Petrie-duality and M\"{o}bius regularity since we know all M\"{o}bius regular maps (with any automorphism group) are also self-Petrie-dual \cite{Wilson}. However, since there are alternative conditions which imply self-Petrie-duality, the converse is not true -- not all self-Petrie-dual regular maps are M\"{o}bius regular.

\bigskip

\section{Regular maps of type $(5,5)$ whose orientation-preserving automorphism group $\langle R, S \rangle$ is isomorphic to $A_5$.}\label{A5}

Adrianov's \cite{Ad} enumeration of regular hypermaps on $PSL(2,q)$ includes a constant which deals with the special case which occurs for maps of type $(5,5)$. 

For us to be considering a map of type $(k,l)$ we must have $2k|q\pm1$ and $2l|q\pm1$, and it is known, see \cite{OKing}, that $PSL(2,q)$ has subgroup $A_5$ when $q \equiv \pm1 \mod 10$. The constant in Adrianov's enumeration, which is 2 for maps of type (5,5) and zero otherwise, is subtracted to account for the cases when the group $\langle R,S \rangle$ collapses into the subgroup $A_5 \le PSL(2,q)$.

The following result, with the usual definitions for $\omega_\kappa$ and $\omega_\lambda$, indicates when the orientation-preserving automorphism group of a type $(5,5)$ map is not the linear fractional group that we might expect, and as such addresses an omission in \cite{CPS2}.

\begin{prop}
The group $\langle R, S \rangle$ of a regular map of type (5,5), generated by the representative matrices $R_4$ and $S_4$, is isomorphic to $A_5$ if and only if $\omega_\lambda \neq \omega_\kappa$.
\end{prop}

\begin{proof}
From \cite{Rose} we know a presentation of the group $A_5$ is: $\langle a, b | a^5, b^5, (ab)^2, (a^4b)^3 \rangle$.

Considering the group $\langle R, S | R^5, S^5, (RS)^2, \dots \rangle$, it is clear that this will be isomorphic to $A_5$ if and only if the condition $(R^4 S)^3= I$ is also satified.
This is the case if and only if $R^{-1}S$ has order 3.

$$
R^{-1}S
=\eta \begin{bmatrix} \xi_\kappa^{-1} & 0 \\ 0 & \xi_\kappa \end{bmatrix} \begin{bmatrix} -\omega_\lambda \xi_\kappa^{-1} & -D \\ 1 & \omega_\lambda \xi_\kappa \end{bmatrix}
=\eta \begin{bmatrix} -\omega_\lambda \xi_\kappa^{-2} & -\xi_\kappa^{-1}D \\ \xi_\kappa & \omega_\lambda \xi_\kappa^2  \end{bmatrix}
$$

$$
(R^{-1}S)^3 = 
\eta^3 \begin{bmatrix} \omega_\lambda(2D\xi_\kappa^{-2}-\omega_\lambda^2\xi_\kappa^{-6} - D\xi_\kappa^2) & D\xi_\kappa^{-2}(D\xi_\kappa + \omega_\lambda^2(\xi_\kappa - \xi_\kappa^5 - \xi_\kappa^{-3})) \\ -(D\xi_\kappa + \omega_\lambda^2(\xi_\kappa - \xi_\kappa^5 - \xi_\kappa^{-3})) & \omega_\lambda(D\xi_\kappa^{-2}+\omega_\lambda^2\xi_\kappa^6 - 2D\xi_\kappa^2) \end{bmatrix}
$$

The off diagonal elements are both zero if and only if $D\xi_\kappa + \omega_\lambda^2(\xi_\kappa - \xi_\kappa^5 - \xi_\kappa^{-3}) = 0$. This condition is equivalent to $(\omega_\kappa+2)(\omega_\kappa-2)(1+\omega_\kappa\omega_\lambda)(1-\omega_\kappa\omega_\lambda)=0$ and we know that $\omega_\kappa \neq \pm2$ so long as $k\neq p$.

The leading diagonal entries are equal if and only if $D(\xi_\kappa^{-2} + \xi_\kappa^2) = \omega_\lambda^2(\xi_\kappa^6 + \xi_\kappa^{-6})$. Applying $\xi_\kappa^{10} = 1$ and eliminating $D$ shows this is equivalent to $(\omega_\kappa^2 -4)(\omega_\kappa^2 - \omega_\kappa^2\omega_\lambda^2 + \omega_\lambda^2 -2)=0$

Assume $\omega_\kappa^2\omega_\lambda^2 =1$.
The off-diagonals are clearly zero, and the leading diagonal entries are equal since $(\omega_\kappa^2 - \omega_\kappa^2\omega_\lambda^2 + \omega_\lambda^2 -2) = \omega_\kappa^{-2}(\omega_\kappa^4 - 3\omega_\kappa^2 +1)=0$ is always the case since the expression inside the bracket is the sum of powers of $\xi_\kappa^2$, a 5th root of unity. Then $R^{-1}S$ has order 3.

Conversely, assume $(R^{-1}S)^3=I$. Then we instantly have $\omega_\kappa^2 \omega_\lambda^2 = 1$ since $p \neq 5$.

We conclude that $(R^{-1}S)^3=I$ if and only if $\omega_\kappa\omega_\lambda = \pm1$. By considering the two possible values for $\omega_\kappa$ and $\omega_\lambda$ we can prove that this will happen if and only if $\omega_\kappa \neq \omega_\lambda$.
\end{proof}

\section{Tables of results and comments}

In the following tables, produced using the computer package GAP \cite{GAP}, we list for given $q \le 49$, all the $PSL(2,q)$ maps which have one or more of the properties we have addressed in the paper, the ticks indicating when the map has each property. The elements $\xi_{\kappa}$, $\xi_\lambda$, $\omega_\kappa$ and $\omega_\lambda$ are expressed as powers of a primitive element $\xi$ in the field $GF(q^2)$.  For a given $k,l$, only one map is shown in each equivalence class under the action of the automorphism group.

For interest we also include an enumeration in Table \ref{enum} which shows how many $PSL(2,q)$ maps there are with each of these combinations of properties for $q\le81$.

\section*{Tables}
\begin{longtable}{|ccc|cccc|ccc|}
	\hline
	$q$ & $k$ & $l$ & $\log_{\xi}\xi_{\kappa}$ & $\log_{\xi}\xi_{\lambda}$ & $\log_{\xi}\omega_{\kappa}$ & $\log_{\xi}\omega_{\lambda}$ & SD & SPD & MR \\
	\hline
	\endhead
	\hline
	\endfoot
	5 & 5 & 5 &  &  &  &  & \cmark & \xmark & \xmark\\

	\hline
	7 & 7 & 7 &  &  &  &  & \cmark & \xmark & \xmark\\

	\hline
	9 & 5 & 5 & 8 & 8 & 30 & 30 & \cmark & \xmark & \xmark\\

	\hline
	11 & 5 & 5 & 12 & 12 & 36 & 36 & \cmark & \xmark & \xmark\\
11 & 5 & 5 & 36 & 36 & 24 & 24 & \cmark & \cmark & \xmark\\
11 & 5 & 6 & 12 & 10 & 36 & 108 & \xmark & \cmark & \xmark\\
11 & 6 & 6 & 10 & 10 & 108 & 108 & \cmark & \xmark & \xmark\\
11 & 11 & 11 &  &  &  &  & \cmark & \xmark & \xmark\\

	\hline
	13 & 6 & 6 & 14 & 14 & 112 & 112 & \cmark & \xmark & \xmark\\
13 & 7 & 7 & 12 & 12 & 126 & 126 & \cmark & \xmark & \xmark\\
13 & 7 & 7 & 12 & 60 & 126 & 56 & \xmark & \cmark & \xmark\\
13 & 7 & 7 & 36 & 36 & 70 & 70 & \cmark & \cmark & \xmark\\
13 & 7 & 7 & 60 & 60 & 56 & 56 & \cmark & \xmark & \xmark\\
13 & 7 & 13 & 60 &  & 56 &  & \xmark & \cmark & \xmark\\
13 & 13 & 13 &  &  &  &  & \cmark & \xmark & \xmark\\

	\hline
	17 & 8 & 8 & 18 & 18 & 36 & 36 & \cmark & \xmark & \xmark\\
17 & 8 & 8 & 54 & 54 & 90 & 90 & \cmark & \xmark & \xmark\\
17 & 8 & 9 & 54 & 80 & 90 & 54 & \xmark & \cmark & \cmark\\
17 & 8 & 17 & 18 &  & 36 &  & \xmark & \cmark & \cmark\\
17 & 9 & 9 & 16 & 16 & 216 & 216 & \cmark & \xmark & \xmark\\
17 & 9 & 9 & 80 & 80 & 54 & 54 & \cmark & \xmark & \xmark\\
17 & 9 & 9 & 112 & 112 & 18 & 18 & \cmark & \xmark & \xmark\\
17 & 17 & 17 &  &  &  &  & \cmark & \xmark & \xmark\\

	\hline
	19 & 3 & 9 & 60 & 20 & 0 & 300 & \xmark & \cmark & \xmark\\
19 & 5 & 5 & 36 & 36 & 320 & 320 & \cmark & \xmark & \xmark\\
19 & 5 & 5 & 108 & 108 & 220 & 220 & \cmark & \xmark & \xmark\\
19 & 9 & 5 & 20 & 36 & 300 & 320 & \xmark & \cmark & \xmark\\
19 & 9 & 9 & 20 & 20 & 300 & 300 & \cmark & \xmark & \xmark\\
19 & 9 & 9 & 100 & 100 & 80 & 80 & \cmark & \xmark & \xmark\\
19 & 9 & 9 & 140 & 100 & 340 & 80 & \xmark & \cmark & \xmark\\
19 & 9 & 9 & 140 & 140 & 340 & 340 & \cmark & \xmark & \xmark\\
19 & 9 & 10 & 100 & 18 & 80 & 60 & \xmark & \cmark & \xmark\\
19 & 10 & 10 & 18 & 18 & 60 & 60 & \cmark & \xmark & \xmark\\
19 & 10 & 10 & 54 & 54 & 100 & 100 & \cmark & \xmark & \xmark\\
19 & 19 & 19 &  &  &  &  & \cmark & \xmark & \xmark\\

	\hline
	23 & 3 & 11 & 88 & 72 & 0 & 168 & \xmark & \cmark & \xmark\\
23 & 6 & 6 & 44 & 44 & 192 & 192 & \cmark & \xmark & \xmark\\
23 & 6 & 11 & 44 & 216 & 192 & 240 & \xmark & \cmark & \cmark\\
23 & 11 & 11 & 24 & 24 & 360 & 360 & \cmark & \xmark & \xmark\\
23 & 11 & 11 & 72 & 72 & 168 & 168 & \cmark & \xmark & \xmark\\
23 & 11 & 11 & 120 & 120 & 480 & 480 & \cmark & \xmark & \xmark\\
23 & 11 & 11 & 168 & 168 & 336 & 336 & \cmark & \xmark & \xmark\\
23 & 11 & 11 & 216 & 216 & 240 & 240 & \cmark & \xmark & \xmark\\
23 & 12 & 11 & 22 & 24 & 144 & 360 & \xmark & \cmark & \cmark\\
23 & 12 & 12 & 22 & 22 & 144 & 144 & \cmark & \xmark & \xmark\\
23 & 12 & 12 & 110 & 110 & 384 & 384 & \cmark & \cmark & \cmark\\
23 & 23 & 23 &  &  &  &  & \cmark & \xmark & \xmark\\

	\hline
	25 & 3 & 13 & 104 & 72 & 0 & 494 & \xmark & \cmark & \xmark\\
25 & 4 & 13 & 78 & 168 & 390 & 26 & \xmark & \cmark & \xmark\\
25 & 6 & 13 & 52 & 24 & 546 & 260 & \xmark & \cmark & \xmark\\
25 & 12 & 12 & 26 & 26 & 416 & 416 & \cmark & \xmark & \xmark\\
25 & 12 & 13 & 26 & 216 & 416 & 130 & \xmark & \cmark & \cmark\\
25 & 13 & 13 & 24 & 24 & 260 & 260 & \cmark & \xmark & \xmark\\
25 & 13 & 13 & 24 & 120 & 260 & 52 & \cmark & \xmark & \xmark\\
25 & 13 & 13 & 72 & 72 & 494 & 494 & \cmark & \xmark & \xmark\\
25 & 13 & 13 & 72 & 264 & 494 & 598 & \cmark & \xmark & \xmark\\
25 & 13 & 13 & 168 & 168 & 26 & 26 & \cmark & \xmark & \xmark\\
25 & 13 & 13 & 168 & 216 & 26 & 130 & \cmark & \xmark & \xmark\\

	\hline
	27 & 7 & 7 & 52 & 52 & 420 & 420 & \cmark & \xmark & \xmark\\
27 & 13 & 7 & 140 & 156 & 28 & 532 & \xmark & \cmark & \xmark\\
27 & 13 & 13 & 28 & 28 & 560 & 560 & \cmark & \xmark & \xmark\\
27 & 13 & 13 & 28 & 252 & 560 & 672 & \xmark & \cmark & \xmark\\
27 & 13 & 13 & 140 & 140 & 28 & 28 & \cmark & \xmark & \xmark\\
27 & 14 & 14 & 26 & 26 & 476 & 476 & \cmark & \xmark & \xmark\\

	\hline
	29 & 3 & 15 & 140 & 28 & 0 & 480 & \xmark & \cmark & \xmark\\
29 & 5 & 5 & 84 & 84 & 180 & 180 & \cmark & \xmark & \xmark\\
29 & 5 & 5 & 252 & 252 & 240 & 240 & \cmark & \xmark & \xmark\\
29 & 5 & 14 & 84 & 90 & 180 & 270 & \xmark & \cmark & \xmark\\
29 & 5 & 29 & 252 &  & 240 &  & \xmark & \cmark & \xmark\\
29 & 7 & 7 & 60 & 60 & 570 & 570 & \cmark & \xmark & \xmark\\
29 & 7 & 7 & 180 & 180 & 750 & 750 & \cmark & \xmark & \xmark\\
29 & 7 & 7 & 300 & 300 & 780 & 780 & \cmark & \xmark & \xmark\\
29 & 14 & 14 & 30 & 30 & 630 & 630 & \cmark & \xmark & \xmark\\
29 & 14 & 14 & 90 & 90 & 270 & 270 & \cmark & \xmark & \xmark\\
29 & 14 & 14 & 150 & 150 & 540 & 540 & \cmark & \xmark & \xmark\\
29 & 15 & 7 & 308 & 180 & 300 & 750 & \xmark & \cmark & \xmark\\
29 & 15 & 7 & 364 & 300 & 810 & 780 & \xmark & \cmark & \xmark\\
29 & 15 & 14 & 196 & 30 & 90 & 630 & \xmark & \cmark & \xmark\\
29 & 15 & 15 & 28 & 28 & 480 & 480 & \cmark & \xmark & \xmark\\
29 & 15 & 15 & 28 & 308 & 480 & 300 & \xmark & \cmark & \xmark\\
29 & 15 & 15 & 196 & 196 & 90 & 90 & \cmark & \xmark & \xmark\\
29 & 15 & 15 & 308 & 308 & 300 & 300 & \cmark & \xmark & \xmark\\
29 & 15 & 15 & 364 & 364 & 810 & 810 & \cmark & \xmark & \xmark\\
29 & 29 & 29 &  &  &  &  & \cmark & \xmark & \xmark\\

	\hline
	31 & 5 & 5 & 96 & 96 & 128 & 128 & \cmark & \xmark & \xmark\\
31 & 5 & 5 & 288 & 288 & 352 & 352 & \cmark & \xmark & \xmark\\
31 & 8 & 8 & 60 & 60 & 160 & 160 & \cmark & \xmark & \xmark\\
31 & 8 & 8 & 180 & 180 & 704 & 704 & \cmark & \xmark & \xmark\\
31 & 8 & 15 & 60 & 352 & 160 & 320 & \xmark & \cmark & \cmark\\
31 & 8 & 16 & 180 & 30 & 704 & 256 & \xmark & \cmark & \cmark\\
31 & 15 & 15 & 32 & 32 & 416 & 416 & \cmark & \xmark & \xmark\\
31 & 15 & 15 & 224 & 224 & 512 & 512 & \cmark & \xmark & \xmark\\
31 & 15 & 15 & 352 & 352 & 320 & 320 & \cmark & \xmark & \xmark\\
31 & 15 & 15 & 416 & 416 & 672 & 672 & \cmark & \xmark & \xmark\\
31 & 16 & 5 & 210 & 288 & 448 & 352 & \xmark & \cmark & \cmark\\
31 & 16 & 8 & 90 & 60 & 576 & 160 & \xmark & \cmark & \cmark\\
31 & 16 & 15 & 150 & 416 & 544 & 672 & \xmark & \cmark & \cmark\\
31 & 16 & 16 & 30 & 30 & 256 & 256 & \cmark & \xmark & \xmark\\
31 & 16 & 16 & 30 & 150 & 256 & 544 & \xmark & \cmark & \cmark\\
31 & 16 & 16 & 90 & 90 & 576 & 576 & \cmark & \xmark & \xmark\\
31 & 16 & 16 & 150 & 150 & 544 & 544 & \cmark & \xmark & \xmark\\
31 & 16 & 16 & 210 & 210 & 448 & 448 & \cmark & \xmark & \xmark\\
31 & 31 & 31 &  &  &  &  & \cmark & \xmark & \xmark\\

	\hline
	37 & 6 & 6 & 114 & 114 & 1178 & 1178 & \cmark & \xmark & \xmark\\
37 & 9 & 9 & 76 & 76 & 190 & 190 & \cmark & \xmark & \xmark\\
37 & 9 & 9 & 380 & 380 & 1292 & 1292 & \cmark & \xmark & \xmark\\
37 & 9 & 9 & 532 & 532 & 1254 & 1254 & \cmark & \xmark & \xmark\\
37 & 18 & 18 & 38 & 38 & 836 & 836 & \cmark & \xmark & \xmark\\
37 & 18 & 18 & 190 & 190 & 266 & 266 & \cmark & \xmark & \xmark\\
37 & 18 & 18 & 266 & 266 & 76 & 76 & \cmark & \xmark & \xmark\\
37 & 19 & 9 & 180 & 380 & 798 & 1292 & \xmark & \cmark & \xmark\\
37 & 19 & 9 & 468 & 76 & 646 & 190 & \xmark & \cmark & \xmark\\
37 & 19 & 18 & 252 & 38 & 1140 & 836 & \xmark & \cmark & \xmark\\
37 & 19 & 18 & 612 & 266 & 988 & 76 & \xmark & \cmark & \xmark\\
37 & 19 & 19 & 36 & 36 & 1026 & 1026 & \cmark & \xmark & \xmark\\
37 & 19 & 19 & 36 & 108 & 1026 & 418 & \xmark & \cmark & \xmark\\
37 & 19 & 19 & 108 & 108 & 418 & 418 & \cmark & \xmark & \xmark\\
37 & 19 & 19 & 108 & 252 & 418 & 1140 & \xmark & \cmark & \xmark\\
37 & 19 & 19 & 180 & 180 & 798 & 798 & \cmark & \xmark & \xmark\\
37 & 19 & 19 & 252 & 252 & 1140 & 1140 & \cmark & \xmark & \xmark\\
37 & 19 & 19 & 324 & 324 & 1064 & 1064 & \cmark & \xmark & \xmark\\
37 & 19 & 19 & 396 & 396 & 228 & 228 & \cmark & \cmark & \xmark\\
37 & 19 & 19 & 468 & 468 & 646 & 646 & \cmark & \xmark & \xmark\\
37 & 19 & 19 & 540 & 324 & 532 & 1064 & \xmark & \cmark & \xmark\\
37 & 19 & 19 & 540 & 540 & 532 & 532 & \cmark & \xmark & \xmark\\
37 & 19 & 19 & 612 & 612 & 988 & 988 & \cmark & \xmark & \xmark\\
37 & 19 & 37 & 324 &  & 1064 &  & \xmark & \cmark & \xmark\\
37 & 37 & 37 &  &  &  &  & \cmark & \xmark & \xmark\\

	\hline
	41 & 5 & 5 & 168 & 168 & 1638 & 1638 & \cmark & \xmark & \xmark\\
41 & 5 & 5 & 504 & 504 & 882 & 882 & \cmark & \xmark & \xmark\\
41 & 5 & 7 & 168 & 600 & 1638 & 126 & \xmark & \cmark & \xmark\\
41 & 5 & 20 & 504 & 294 & 882 & 924 & \xmark & \cmark & \xmark\\
41 & 7 & 7 & 120 & 120 & 210 & 210 & \cmark & \xmark & \xmark\\
41 & 7 & 7 & 360 & 360 & 504 & 504 & \cmark & \xmark & \xmark\\
41 & 7 & 7 & 600 & 600 & 126 & 126 & \cmark & \xmark & \xmark\\
41 & 10 & 10 & 84 & 84 & 630 & 630 & \cmark & \xmark & \xmark\\
41 & 10 & 10 & 252 & 252 & 672 & 672 & \cmark & \xmark & \xmark\\
41 & 10 & 21 & 84 & 520 & 630 & 336 & \xmark & \cmark & \cmark\\
41 & 10 & 41 & 252 &  & 672 &  & \xmark & \cmark & \cmark\\
41 & 20 & 20 & 42 & 42 & 1302 & 1302 & \cmark & \xmark & \xmark\\
41 & 20 & 20 & 42 & 126 & 1302 & 714 & \xmark & \cmark & \cmark\\
41 & 20 & 20 & 126 & 126 & 714 & 714 & \cmark & \xmark & \xmark\\
41 & 20 & 20 & 294 & 294 & 924 & 924 & \cmark & \xmark & \xmark\\
41 & 20 & 20 & 378 & 378 & 420 & 420 & \cmark & \xmark & \xmark\\
41 & 20 & 21 & 126 & 680 & 714 & 1218 & \xmark & \cmark & \cmark\\
41 & 20 & 21 & 294 & 200 & 924 & 168 & \xmark & \cmark & \cmark\\
41 & 20 & 21 & 378 & 440 & 420 & 756 & \xmark & \cmark & \cmark\\
41 & 21 & 21 & 40 & 40 & 1428 & 1428 & \cmark & \xmark & \xmark\\
41 & 21 & 21 & 200 & 200 & 168 & 168 & \cmark & \xmark & \xmark\\
41 & 21 & 21 & 440 & 440 & 756 & 756 & \cmark & \xmark & \xmark\\
41 & 21 & 21 & 520 & 520 & 336 & 336 & \cmark & \xmark & \xmark\\
41 & 21 & 21 & 680 & 680 & 1218 & 1218 & \cmark & \xmark & \xmark\\
41 & 21 & 21 & 760 & 760 & 1134 & 1134 & \cmark & \xmark & \xmark\\
41 & 41 & 41 &  &  &  &  & \cmark & \xmark & \xmark\\

	\hline
	43 & 3 & 22 & 308 & 294 & 0 & 1276 & \xmark & \cmark & \xmark\\
43 & 7 & 7 & 132 & 132 & 792 & 792 & \cmark & \xmark & \xmark\\
43 & 7 & 7 & 396 & 396 & 220 & 220 & \cmark & \xmark & \xmark\\
43 & 7 & 7 & 660 & 660 & 1760 & 1760 & \cmark & \xmark & \xmark\\
43 & 7 & 21 & 132 & 484 & 792 & 1628 & \xmark & \cmark & \xmark\\
43 & 7 & 21 & 396 & 572 & 220 & 484 & \xmark & \cmark & \xmark\\
43 & 7 & 21 & 660 & 748 & 1760 & 1496 & \xmark & \cmark & \xmark\\
43 & 11 & 11 & 84 & 84 & 1452 & 1452 & \cmark & \xmark & \xmark\\
43 & 11 & 11 & 252 & 252 & 1012 & 1012 & \cmark & \xmark & \xmark\\
43 & 11 & 11 & 420 & 420 & 1540 & 1540 & \cmark & \xmark & \xmark\\
43 & 11 & 11 & 588 & 588 & 1584 & 1584 & \cmark & \xmark & \xmark\\
43 & 11 & 11 & 756 & 756 & 1804 & 1804 & \cmark & \xmark & \xmark\\
43 & 21 & 7 & 484 & 660 & 1628 & 1760 & \xmark & \cmark & \xmark\\
43 & 21 & 11 & 44 & 420 & 396 & 1540 & \xmark & \cmark & \xmark\\
43 & 21 & 11 & 220 & 588 & 1100 & 1584 & \xmark & \cmark & \xmark\\
43 & 21 & 11 & 748 & 84 & 1496 & 1452 & \xmark & \cmark & \xmark\\
43 & 21 & 11 & 836 & 252 & 440 & 1012 & \xmark & \cmark & \xmark\\
43 & 21 & 21 & 44 & 44 & 396 & 396 & \cmark & \xmark & \xmark\\
43 & 21 & 21 & 220 & 220 & 1100 & 1100 & \cmark & \xmark & \xmark\\
43 & 21 & 21 & 484 & 484 & 1628 & 1628 & \cmark & \xmark & \xmark\\
43 & 21 & 21 & 572 & 220 & 484 & 1100 & \xmark & \cmark & \xmark\\
43 & 21 & 21 & 572 & 572 & 484 & 484 & \cmark & \xmark & \xmark\\
43 & 21 & 21 & 748 & 748 & 1496 & 1496 & \cmark & \xmark & \xmark\\
43 & 21 & 21 & 836 & 836 & 440 & 440 & \cmark & \xmark & \xmark\\
43 & 22 & 22 & 42 & 42 & 132 & 132 & \cmark & \xmark & \xmark\\
43 & 22 & 22 & 126 & 126 & 308 & 308 & \cmark & \xmark & \xmark\\
43 & 22 & 22 & 210 & 210 & 968 & 968 & \cmark & \xmark & \xmark\\
43 & 22 & 22 & 294 & 294 & 1276 & 1276 & \cmark & \xmark & \xmark\\
43 & 22 & 22 & 378 & 378 & 1672 & 1672 & \cmark & \xmark & \xmark\\
43 & 43 & 43 &  &  &  &  & \cmark & \xmark & \xmark\\

	\hline
	47 & 3 & 23 & 368 & 528 & 0 & 1152 & \xmark & \cmark & \xmark\\
47 & 6 & 6 & 184 & 184 & 480 & 480 & \cmark & \xmark & \xmark\\
47 & 6 & 24 & 184 & 46 & 480 & 672 & \xmark & \cmark & \cmark\\
47 & 8 & 8 & 138 & 138 & 960 & 960 & \cmark & \xmark & \xmark\\
47 & 8 & 8 & 414 & 414 & 576 & 576 & \cmark & \xmark & \xmark\\
47 & 8 & 23 & 138 & 624 & 960 & 144 & \xmark & \cmark & \cmark\\
47 & 8 & 23 & 414 & 432 & 576 & 528 & \xmark & \cmark & \cmark\\
47 & 12 & 8 & 92 & 414 & 1200 & 576 & \xmark & \cmark & \cmark\\
47 & 12 & 12 & 92 & 92 & 1200 & 1200 & \cmark & \xmark & \xmark\\
47 & 12 & 12 & 460 & 460 & 1008 & 1008 & \cmark & \xmark & \xmark\\
47 & 12 & 23 & 460 & 816 & 1008 & 768 & \xmark & \cmark & \cmark\\
47 & 23 & 23 & 48 & 48 & 1344 & 1344 & \cmark & \xmark & \xmark\\
47 & 23 & 23 & 144 & 144 & 1056 & 1056 & \cmark & \xmark & \xmark\\
47 & 23 & 23 & 240 & 240 & 288 & 288 & \cmark & \xmark & \xmark\\
47 & 23 & 23 & 336 & 336 & 720 & 720 & \cmark & \xmark & \xmark\\
47 & 23 & 23 & 432 & 432 & 528 & 528 & \cmark & \xmark & \xmark\\
47 & 23 & 23 & 528 & 528 & 1152 & 1152 & \cmark & \xmark & \xmark\\
47 & 23 & 23 & 624 & 624 & 144 & 144 & \cmark & \xmark & \xmark\\
47 & 23 & 23 & 720 & 720 & 1296 & 1296 & \cmark & \xmark & \xmark\\
47 & 23 & 23 & 816 & 816 & 768 & 768 & \cmark & \xmark & \xmark\\
47 & 23 & 23 & 912 & 912 & 624 & 624 & \cmark & \xmark & \xmark\\
47 & 23 & 23 & 1008 & 1008 & 2016 & 2016 & \cmark & \xmark & \xmark\\
47 & 24 & 12 & 46 & 460 & 672 & 1008 & \xmark & \cmark & \cmark\\
47 & 24 & 23 & 322 & 144 & 1920 & 1056 & \xmark & \cmark & \cmark\\
47 & 24 & 23 & 506 & 48 & 336 & 1344 & \xmark & \cmark & \cmark\\
47 & 24 & 24 & 46 & 46 & 672 & 672 & \cmark & \xmark & \xmark\\
47 & 24 & 24 & 230 & 230 & 1488 & 1488 & \cmark & \cmark & \cmark\\
47 & 24 & 24 & 322 & 322 & 1920 & 1920 & \cmark & \xmark & \xmark\\
47 & 24 & 24 & 506 & 506 & 336 & 336 & \cmark & \xmark & \xmark\\
47 & 47 & 47 &  &  &  &  & \cmark & \xmark & \xmark\\

	\hline
	49 & 4 & 25 & 300 & 144 & 400 & 1250 & \xmark & \cmark & \xmark\\
49 & 5 & 5 & 240 & 240 & 1350 & 1350 & \cmark & \xmark & \xmark\\
49 & 6 & 24 & 200 & 250 & 1400 & 2050 & \xmark & \cmark & \xmark\\
49 & 7 & 24 &  & 50 &  & 500 & \xmark & \cmark & \xmark\\
49 & 7 & 25 &  & 432 &  & 1500 & \xmark & \cmark & \xmark\\
49 & 7 & 25 &  & 816 &  & 100 & \xmark & \cmark & \xmark\\
49 & 8 & 25 & 150 & 48 & 2200 & 1450 & \xmark & \cmark & \xmark\\
49 & 8 & 25 & 450 & 528 & 600 & 950 & \xmark & \cmark & \xmark\\
49 & 12 & 12 & 100 & 100 & 750 & 750 & \cmark & \xmark & \xmark\\
49 & 12 & 25 & 100 & 144 & 750 & 1250 & \xmark & \cmark & \cmark\\
49 & 24 & 12 & 50 & 100 & 500 & 750 & \xmark & \cmark & \cmark\\
49 & 24 & 24 & 50 & 50 & 500 & 500 & \cmark & \xmark & \xmark\\
49 & 24 & 24 & 50 & 350 & 500 & 1100 & \cmark & \xmark & \xmark\\
49 & 24 & 24 & 250 & 250 & 2050 & 2050 & \cmark & \xmark & \xmark\\
49 & 24 & 24 & 250 & 550 & 2050 & 1150 & \cmark & \xmark & \xmark\\
49 & 24 & 25 & 250 & 912 & 2050 & 700 & \xmark & \cmark & \cmark\\
49 & 25 & 25 & 48 & 48 & 1450 & 1450 & \cmark & \xmark & \xmark\\
49 & 25 & 25 & 48 & 336 & 1450 & 550 & \cmark & \xmark & \xmark\\
49 & 25 & 25 & 144 & 144 & 1250 & 1250 & \cmark & \xmark & \xmark\\
49 & 25 & 25 & 144 & 1008 & 1250 & 1550 & \cmark & \xmark & \xmark\\
49 & 25 & 25 & 432 & 432 & 1500 & 1500 & \cmark & \xmark & \xmark\\
49 & 25 & 25 & 432 & 624 & 1500 & 900 & \cmark & \xmark & \xmark\\
49 & 25 & 25 & 528 & 528 & 950 & 950 & \cmark & \xmark & \xmark\\
49 & 25 & 25 & 528 & 1104 & 950 & 1850 & \cmark & \xmark & \xmark\\
49 & 25 & 25 & 816 & 816 & 100 & 100 & \cmark & \xmark & \xmark\\
49 & 25 & 25 & 816 & 912 & 100 & 700 & \cmark & \xmark & \xmark\\

	\hline
\end{longtable}

\smallskip

\begin{table}
\begin{tabular}{|b{0.5cm}b{1cm}|*{5}{>{~~}b{1.3cm}}b{1.9cm}|}
	\hline
	$q$ & Maps & None & SD only & SP only & SD+SP & SP+MR & SD+SP+MR \tabularnewline
	\hline
		5 & 1 & 0 & 1 & 0 & 0 & 0 & 0 \tabularnewline
	7 & 5 & 4 & 1 & 0 & 0 & 0 & 0 \tabularnewline
	9 & 3 & 2 & 1 & 0 & 0 & 0 & 0 \tabularnewline
	11 & 16 & 11 & 3 & 1 & 1 & 0 & 0 \tabularnewline
	13 & 33 & 26 & 4 & 2 & 1 & 0 & 0 \tabularnewline
	17 & 58 & 50 & 6 & 0 & 0 & 2 & 0 \tabularnewline
	19 & 70 & 58 & 8 & 4 & 0 & 0 & 0 \tabularnewline
	23 & 113 & 101 & 8 & 1 & 0 & 2 & 1 \tabularnewline
	25 & 63 & 52 & 7 & 3 & 0 & 1 & 0 \tabularnewline
	27 & 54 & 48 & 4 & 2 & 0 & 0 & 0 \tabularnewline
	29 & 183 & 163 & 13 & 7 & 0 & 0 & 0 \tabularnewline
	31 & 209 & 190 & 13 & 0 & 0 & 6 & 0 \tabularnewline
	37 & 315 & 290 & 16 & 8 & 1 & 0 & 0 \tabularnewline
	41 & 382 & 356 & 18 & 2 & 0 & 6 & 0 \tabularnewline
	43 & 430 & 400 & 20 & 10 & 0 & 0 & 0 \tabularnewline
	47 & 515 & 485 & 20 & 1 & 0 & 8 & 1 \tabularnewline
	49 & 264 & 238 & 16 & 7 & 0 & 3 & 0 \tabularnewline
	53 & 663 & 625 & 25 & 13 & 0 & 0 & 0 \tabularnewline
	59 & 820 & 779 & 27 & 13 & 1 & 0 & 0 \tabularnewline
	61 & 879 & 836 & 28 & 14 & 1 & 0 & 0 \tabularnewline
	67 & 1072 & 1024 & 32 & 16 & 0 & 0 & 0 \tabularnewline
	71 & 1199 & 1151 & 32 & 4 & 0 & 11 & 1 \tabularnewline
	73 & 1276 & 1227 & 33 & 3 & 1 & 12 & 0 \tabularnewline
	79 & 1493 & 1438 & 37 & 2 & 0 & 16 & 0 \tabularnewline
	81 & 381 & 356 & 15 & 7 & 0 & 3 & 0 \tabularnewline

	\hline
\end{tabular}
\caption{External symmetries of regular maps on $PSL(2,q)$}\label{enum}
\end{table}

\newpage

It has been an open problem for some time as to whether there exists a self-dual and self-Petrie-dual regular map for any given vertex degree $k$ on some surface. In \cite{ACS}, Archdeacon, Conder and \v{S}ir\'{a}\v{n} proved the existence of such a map for any even valency.
The work in this paper allows us to prove existence of a self-dual, Self-Petrie-dual regular map for any given odd valency $k\ge5$. This will be proved in an upcoming paper by Fraser, Jeans and \v{S}ir\'{a}\v{n} \cite{FJS}.

\smallskip

\bigskip
{\bf Acknowledgements:} The authors would like to thank Jozef \v{S}ir\'{a}\v{n} for many useful disussions during the preparation of this paper. The second author acknowledges partial support by Slovak research grants VEGA 1/0026/16, VEGA 1/0142/17, APVV-0136-12, APVV-15-0220 and APVV-17-0428.


\end{document}